\documentclass{amsart}
\usepackage{amssymb}
\usepackage{amsmath}
\usepackage{amscd}
\usepackage{graphicx}
\usepackage[all]{xy}
\usepackage[T1]{fontenc}
\newtheorem{theorem}{Theorem}[section]

\newtheorem{lemma}[theorem]{Lemma}
\newtheorem{proposition}[theorem]{Proposition}
\theoremstyle{definition}

\newtheorem{subsec}[theorem]{}

\newcommand{\Br}{\operatorname{Br}}
\newcommand{\Ker}{\operatorname{Ker}}
\newcommand{\Tr}{\operatorname{Tr}}

\begin{document}

\title[Clifford extensions of blocks]{$G$-algebras, group graded algebras, \\ and Clifford extensions of blocks}
\author[Tiberiu Cocone\c t]{Tiberiu Cocone\c t \\ \textrm{Babe\c s-Bolyai University} \\ Faculty of Mathematics and Computer Science \\ Str. Mihail Kog\u alniceanu 1 \\ 400084 Cluj-Napoca, Romania\\
 e-mail: \text{tiberiu.coconet@math.ubbcluj.ro}}
\date{}

{\abstract   Let  $K$ be a normal subgroup of the finite group $H$. To a block of a $K$-interior $H$-algebra we  associate a group extension, and we
prove that this extension is isomorphic to an extension associated to a  block given  by the  Brauer homomorphism.
 This may be regarded as a generalization and an alternative treatment of Dade's results \cite[Section 12]{Dade}.    \endabstract}



\maketitle


\section{Introduction}

\begin{subsec} Let $p$ be a prime, and let $\mathcal{O}$ be a discrete valuation ring  with residue field $k$ of characteristic $p$. We make no
assumptions on the size of $\mathcal{O}$  and $k$, also allowing $\mathcal{O}=k.$  Let $K$ be a normal subgroup of the finite
group $H$, and we denote by $G$ the factor group $H/K$. \end{subsec}

\begin{subsec} Under the assumption that $k$ is algebraically closed, Dade  introduced in \cite{Dade}
the Clifford extension of a block $b$ of $K$, and proved that this extension (which is a central extension of $k^*$ by a certain stabilizer of $b$)
is isomorphic to an extension associated to an irreducible modular character lying in a Brauer
correspondent of $b$ (see \cite[Section 12]{Dade}). Let us briefly discuss the arguments and assumptions
used in \cite{Dade}. Two main  ingredients for the isomorphism of the Clifford extensions were \cite[Theorems 8.7 and 9.5]{Dade}.
 Note that \cite[Theorem 8.7]{Dade} is essentially Brauer's First Main Theorem, while \cite[Theorem 9.5]{Dade} extends
 the Brauer correspondence to the case of conjugacy classes of maximal ideals.  Besides these two theorems
  there are two important sets of hypotheses, namely \cite[7.1]{Dade} and \cite[10.1]{Dade}. These
   conditions assure that the equalities \cite[10.4]{Dade} hold, and give a situation which is a slightly
   more general case than of the group ring $\mathcal{O}H$.
\end{subsec}

\begin{subsec} Here we consider a unitary $K$-interior $H$-algebra $A_1$ over  $\mathcal{O}$. The $K$-interior $H$-algebra $A_1$ gives rise to an $H$-interior algebra $A$, which is strongly $G$-graded (see \cite[9.1]{Puig}, or \cite[2.1 and 2.2]{DicuMarcus} for a more general version of this construction). The $H$-interior algebra $A$ does not satisfy conditions \cite[7.1]{Dade}, and therefore the equality \cite[10.4 a)]{Dade}, \[C_A(A_1)=A^K,\]  it is also not satisfied. Still, one of the centralizers occuring in this equality is suitable for the construction of the  extensions. We  will not use a mix of two results analogous to \cite[Theorems 8.7 and 9.5]{Dade}, because for that we would need more assumptions. Instead, we consider two blocks that arise from  \cite[Theorem 3.2]{Barker}.
\end{subsec}

\begin{subsec} We start by choosing  a  primitive idempotent $b$ of $A_1^K$ which lies in the center of $A_1$. By the assumptions made on $A_1$, the idempotent $b$ is actually a block of $A_1$.  Instead of working with the centralizer $C_A(A_1)$ we consider $A^K$. As we have already mentioned, the equality $C_A(A_1)=A^K$ holds if the algebra $A$ satisfies assumptions \cite[7.1]{Dade}. In our case we only have the inclusion \[C_A(A_1)\subseteq A^K=C_A(K\cdot 1),\]  because $A_1$ is $K$-interior.  We work here with $A^K$,  and we consider, as in \cite[Paragraph 2]{Dade} the subgroup $G_b$ of $G$ fixing $b$, and its normal subgroup $G[b]$. Restricting $A^K$ to the components indexed by $G[b]$ gives a strongly $G[b]$-graded $G_b$-algebra, and this allows the construction in Section \ref{s:firs-ext} of the Clifford extension associated  to the block $b.$
\end{subsec}

\begin{subsec} Let $P\le K$ be a  defect group  of the point $\{b\}$ of $A_1^K$. The Brauer quotient $A(P)$ of
 the algebra $A$ with respect to $P$ is a $G$-graded algebra,  and its restriction $A'(P)$ to $C_H(P)K/K$ is strongly graded.
  The Brauer quotient $A_1(P)$ of $A_1$ is a $C_K(P)$-interior $C_H(P)$-algebra, and, as in \cite[2.1]{DicuMarcus} again,
   it gives rise to a $C_H(P)$-interior $C_H(P)/C_K(P)$-graded algebra $R$. We show in Section \ref{s:brauer} that $R$ and $A'(P)$ are
   isomorphic as $N_H(P)$-algebras.

By  \cite[Theorem 3.2]{Barker}, to the point $\{b \}$ of $A_1^K$ it corresponds the point $\{\bar{b}:=\Br_P(b)\}$ of $A_1(P)^{N_K(P)}.$
Since $b$ is central then so is $\bar{b}.$ Using the isomorphism stated in Proposition \ref{brauer:iso}, we regard $\bar{b}$ as a primitive
 idempotent of $R_1^{N_K(P)}$. Next, we consider the centralizer \[C_R(C_K(P))^{N_K(P)}=R^{N_K(P)}.\]
This is a $C_H(P)/C_K(P)$-graded $N_H(P)$-algebra, and we localize it by using $\bar{b}$  to the subgroup $C_H(P)_{\bar{b}}/C_K(P).$
 Further, we take the $N_K(P)$-fixed elements  and we consider the group $\overline{C_H(P)}[\bar{b}]$  for which $\bar{b}R^{N_K(P)}$ is strongly graded. We construct in Section \ref{s:second-ext} the first Clifford extension associated to $\bar{b}.$
\end{subsec}

\begin{subsec} Of course,  the algebra $A(P)$ is $C_H(P)$-interior. The centralizer
\[C_{A(P)}(C_K(P))=A(P)^{C_K(P)}\]
is too large, so one must take
\[(C_{A(P)}(C_K(P)))^{N_K(P)}=A(P)^{N_K(P)}.\]
Going back to the $G$-graded Brauer quotient we see that
 \[\bar{b}A(P)\bar{b}=\bigoplus_{\sigma\in G_b}\bar{b}A(P)_\sigma.\]
We denote $G[\bar{b}]$ the subgroup of $G_b$ such that $$E:=\bigoplus_{\sigma\in G[\bar{b}]}\bar{b}A(P)_{\sigma}^{N_K(P)}$$ is strongly graded.
Since $E_1=\bar{b}A_1(P)^{N_K(P)}$ is local, we consider the quotient
\[\bar{E}=E/EJ(E_1)\] the crossed product of $\bar{E}_1$ with $G[\bar{b}]$ corresponding to  the second Clifford extension of $\bar{b}.$
\end{subsec}

\begin{subsec}  In \cite[Section 11]{Dade} a so-called ``right kernel" of a bilinear map is  introduced.
That map exists because of the assumption that the field $k$ is algebraically closed,
 so the construction of the  Clifford extensions yields  twisted group algebras over $k$ whose components
 are one dimensional. In the present situation, we see that in general the two Clifford extensions correspond to crossed
 products over the skew fields
\[\bar{C}_1:=bA_1^K/J(bA_1^K)\] and  \[\bar{E}_1:= {\bar{b}}(A_1(P))^{N_K(P)}/J({\bar{b}}(A_1(P))^{N_K(P)})\]
extending  $k$, respectively. Clearly $\bar{C}_1$ and $\bar{E}_1$ are trivially $N_K(P)$-acted so there is no need to work with this subgroup.

Our main result is Theorem \ref{3} below, where we show that the above crossed products containing $\bar{C}_1$ and $\bar{E}_1$
as identity components are isomorphic as
group-graded algebras, and the isomorphism preserves the action of the stabilizer of $b$.
Moreover, we show that $G[b]=G[\bar{b}]$ is a subgroup in the normalizer of $P$ in $H_b$ and that these two isomorphic extensions contain
the first Clifford extension of $\bar{b}.$
\end{subsec}

\begin{subsec} At the end of the paper we deal with the special case of the group algebra. In this case, our $K$-interior $H$-algebra is  $A_1=\mathcal{O}K$, and  we have the equalities $$C_{\mathcal{O}H}(\mathcal{O}K)=(\mathcal{O}H)^K$$  and $$(C_{kC_H(D)}(kC_K(D)))^{N_K(D)}=(kC_H(D))^{N_K(D)}.$$
For this particular $K$-interior $H$-algebra, the result tells more. For both algebras $\mathcal{O}H^P$ and $\mathcal{O}N_H(P)^P$ the Brauer quotient is the same, exactly $kC_H(P).$ So considering again the Clifford extensions as before one ends up with three isomorphic crossed products acted by the same $N_H(P)_b.$ This generalizes \cite[Corollay 12.6]{Dade} (and see also \cite{Dade2} for a concise presentation) to the case of  arbitrary base field $k$. \end{subsec}

Our general assumptions and notations are standard. We refer the reader to \cite{The} and \cite{Puig} for Puig's theory of $G$-algebras and pointed groups, and to \cite{M} for facts on group graded algebras.

\section{The Clifford extension of a block}  \label{s:firs-ext}

\begin{subsec} \label{s:graded-acted} As in the introduction, let $K$ be a normal subgroup   of the finite  group $H$, and let $G=H/K$. Let   $A_1$ be an unitary $K$-interior
 $H$-algebra over the  $\mathcal{O}$. As is \cite[2.1]{DicuMarcus}, there  exists a strongly $G$-graded algebra
 $$A:=\bigoplus_{\sigma\in G}A_{\sigma}$$ with structural homomorphisms  $$\mathcal{O}H\to A$$ of $G$-graded algebras.
 This homomorphism  endows $A$ with the structure of a $G$-graded $H$-interior algebra,  and hence $A$ is also an $H$-algebra by conjugation:
 \[a_{\sigma}^h=h^{-1}\cdot a_{\sigma}\cdot h\in A _{\sigma^h},\]  where  $a_{\sigma}\in A_{\sigma}$, $\sigma\in G$,  and  $h\in H.$
 Moreover, we have that $$A_{\sigma}=A_1\otimes x \mbox{ and } A_1\simeq A_1\otimes 1,$$  for each $\sigma \in G$. Here $x$ is a  representative in $\sigma$ and the
 tensor product is over $k.$  The multiplication in $A$ is defined in \cite[2.1]{DicuMarcus}, and see also \cite[9.1]{Puig}. \end{subsec}


\begin{subsec} {\rm  Let $A_1^K$ denote the subalgebra of $A_1$ consisting of  elements fixed under the conjugation action of $K.$
The interior $H$-algebra is $K$-interior by restriction, hence in the same manner we may consider
\[C_A(K\cdot 1)=A^K,\] the fixed elements in $A$ under the conjugation action of the same group $K$.
 Because $K$ is a normal subgroup of $H$ we clearly have that  \[A^K=\bigoplus_{\sigma \in G}(A_{\sigma})^K\] is a $G$-graded subalgebra of $A$.}\end{subsec}

\begin{subsec}  Let $b$ be a primitive idempotent  of  $A_1^K$, lying in $Z(A_1)$.
 The idempotent $b$ need not necessarily be central in $C_A(K\cdot 1)$,
 so it is convenient to consider the stabilizer $G_b$ of $b$ in $G$.  In this situation, $bA^K$ becomes a $G_b$-graded $G_b$-interior algebra,
 but it is not strongly graded in general.  As in \cite{Dade}, we consider the subset
 \[G[b]=\{\sigma\in G_b\mid (bA_{\sigma})^K\cdot (bA_{\sigma^{-1}})^K=(bA_{1})^K \}.\]  of $G_b$.
 \end{subsec}

\begin{proposition}\label{proposition 2.4}
The subset $G[b]$ is a normal subgroup of $G_b.$
\end{proposition}

\begin{proof} For all $\sigma, \tau\in G[b]$, we have
\begin{align*}(bA_1)^K =& (bA_{\sigma})^K\cdot (bA_{\sigma^{-1}})^K=(bA_{\sigma})^K\cdot (bA_{1})^K \cdot(bA_{\sigma^{-1}})^K   \\
 =& (bA_{\sigma})^K\cdot (bA_{\tau})^K\cdot (bA_{\tau^{-1}})^K \cdot(bA_{\sigma^{-1}})\\      \subseteq &(bA_{\sigma\tau})^K\cdot (bA_{(\sigma\tau)^{-1}})^K.
 \end{align*}
This proves that $G[b]$ is a subgroup of $G_b.$
Now consider the elements $\sigma\in G[b]$ and $\theta \in G_b$.
We shall prove the equality \[((bA_{\sigma})^K)^{\theta} =(bA_{\sigma^{\theta}})^K.\]
For this, it suffices to take $h\in \theta$ and prove that \[((bA_{\sigma})^K)^{h} =(bA_{\sigma^{h}})^K.\]
Let $a_{\sigma}\in (bA_{\sigma})^{K}$, then $a_{\sigma}^h\in A_{\sigma^h}.$  For any $l\in K$
we have $$(a_{\sigma}^h)^l=l^{-1}\cdot a_{\sigma}^h\cdot l=l^{-1}h^{-1}\cdot a_{\sigma}\cdot hl=h^{-1}\cdot a_{\sigma}\cdot h.$$
Conversely, let $a_{\sigma^h}\in (bA_{\sigma^{h}})^K.$ As before, we obtain that $(a_{\sigma^h})^{h^{-1}}\in (bA_{\sigma^{}})^K$,
and by applying the action of $h$, the desired inclusion follows. We have
\[((bA_{\sigma})^K)^{\theta}\cdot ((bA_{\sigma^{-1}})^K)^{\theta}=((bA_1)^K)^{\theta},\] or equivalently,
\[(bA_{\sigma^{\theta}})^K\cdot (bA_{(\sigma^{\theta})^{-1}})^K=(bA_1)^K, \] which proves the statement.
\end{proof}

\begin{subsec} \label{s:c} We now denote \[C:=\bigoplus_{\sigma\in G[b]}b(A_{\sigma})^K=\bigoplus_{\sigma\in G[b]}C_{\sigma}\]
where for each $\sigma\in G[b]$, we have denoted \[C_{\sigma}:=b(A_{\sigma})^K.\] The proposition above implies that $C$ is
a strongly $G[b]$-graded $G_{b}$-algebra. Its identity component \[C_1=bA_1^{K}\] is a local ring,
 so by \cite[Lemma 1.1]{Sch} it follows that $C$ is a crossed product of $C_1$ with $G[b]$.
\end{subsec}

\begin{subsec} \label{s:barC}
We have  the skew field  $$\bar{C}_1:= C_1/J(C_1)$$ whose center $\hat{k}_1=Z(\bar{C}_1)$ is a finite extension of $k$.
 Consequently,  $$\bar{C}:=C/CJ(C_1)$$ is a crossed product, and  in the same time it is a $G_b$-algebra. For any $\sigma\in G[b]$
 we identify $$\bar{C}_{\sigma}=C_{\sigma}/C_{\sigma}J(C_1).$$  The corresponding group extension
 \begin{align*}\tag{1}1\to \bar{C}_1^{*}\to \operatorname{hU}(\bar{C})\to G[b]\to 1\end{align*}
 is the {\it Clifford extension} of the block $b$. Here $\operatorname{hU}$ denotes the group of homogenous units.  \end{subsec}

\section{The Brauer Quotient}  \label{s:brauer}

By the Brauer quotient of a $H$-algebra with respect to a $p$-subgroup $P$ of $H$ we mean the structure of an $N_H(P)$-algebra as presented in  \cite[\S 11]{The}. In general, there is no natural graded structure of the Brauer quotient, with respect to an arbitrary $p$-subgroup of $H$, on our strongly $G$-graded $H$-interior algebra $A$. Therefore, we consider the situation when $P$ is a $p$-subgroup of $K$.

\begin{subsec} Let $P$ be a $p$-subgroup of $K$. Since the restriction to $K$ of the action of $H$ on the algebra $A$  leaves invariant each homogeneous component of $A$, we have $$A^P=\bigoplus_{\sigma\in G}A_{\sigma}^P.$$
For any subgroup $Q$ in $P$ we have  the equality $A_Q^P\cap A_{\sigma}^P=(A_{\sigma})_Q^P.$ Using this equality, the Brauer quotient of $A$  becomes
\begin{align*}A(P) & = k\otimes_{\mathcal{O}} (A^P/\sum_{Q<P}A_Q^P) \\  &= k\otimes_{\mathcal{O}} (\bigoplus_{\sigma\in G}(A_{\sigma})^P)/(\bigoplus_{\sigma\in G}\sum_{Q<P}(A_{\sigma})_Q^P) \\ &\simeq  k\otimes_{\mathcal{O}}\bigoplus_{\sigma\in G}
((A_{\sigma})^P/\sum_{Q<P}(A_{\sigma})_Q^P)= \bigoplus_{\sigma\in G}A(P)_{\sigma.}\end{align*}
In this situation,  the $N_H(P)$-algebra $A(P)$ is $G$-graded, with identity component $A(P)_1\simeq A_1(P).$
\end{subsec}

\begin{subsec}  It is useful to consider the  subalgebra
 $$A'(P)=\bigoplus_{\sigma\in C_H(P)K/K}A(P)_{\sigma},$$ which is an $N_H(P)$-invariant subalgebra of $A(P).$ With the notation of \ref{s:graded-acted}, the  map $a\mapsto a\otimes 1$ is an isomorphism of $H$-algebras between   $A_1$ and $A_1\otimes 1.$ This implies that \[A^P_1\simeq (A_1\otimes 1)^P=A_1^P\otimes 1.\] If $\sigma\in C_H(P)K/K$ and $a\in A_1^P$ and $x\in \sigma\cap C_H(P)$, we can identify $a$ with
$$a\otimes 1=(a\otimes x)(1\otimes x^{-1})\in A_{\sigma}^P\cdot A_{\sigma^{-1}}^P.$$ Then  the equality $A_1^P=A_{\sigma}^P\cdot A_{\sigma^{-1}}^P$
is valid, showing that $A'(P)$ is a strongly $C_H(P)K/K$-graded $N_H(P)$-algebra. In fact, this also follows from the proof of Proposition \ref{brauer:iso}
below.
\end{subsec}

\begin{subsec} \label{s:r} We may also construct  the Brauer quotient $A_1(P)$  of the  the $K$-interior $H$-algebra
 $A_1$, so $A_1(P)$ is a $C_K(P)$-interior $N_H(P)$-algebra. We then regard  $A_1(P)$ to be a $C_K(P)$-interior $C_H(P)$-algebra. In this case, following
 \cite[2.1]{DicuMarcus} we obtain the strongly $C_H(P)/C_K(P)$-graded $N_H(P)$-interior $k$-algebra
 \[R:=\bigoplus_{\tau\in C_H(P)/C_K(P)}R_{\tau}.\] For each $\tau\in C_H(P)/C_K(P)$ we have  \[R_{\tau}=A_1(P)\otimes x\] for some representative $x$ in $\tau$.
 \end{subsec}

\begin{proposition}\label{brauer:iso} The strongly graded $k$-algebras $A'(P)$ and $R$ are isomorphic as $N_H(P)$-algebras.
 \end{proposition}

 \begin{proof} There is an obvious bijection between the sets of components of the two algebras, because of the natural isomorphism
 \[C_H(P)K/K\simeq C_H(P)/C_K(P).\]
 Next, we fix $\tau \in C_H(P)/C_K(P)$ and $\sigma\in C_H(P)K/K$ such that $\tau=\sigma\cap C_H(P)$, and we define
 \[\psi_{\tau}:R_{\tau}\to A'(P)_{\sigma}, \qquad \bar{a}\otimes x \mapsto \overline{a\otimes x},\]   for  $x\in \tau$.
 We have denoted  $\bar{a}:=\Br_P(a)$ for some element $a$ in $A_1^P$, while $x$ stands for a representative in $\tau$.
 If $a_1\in \bar{a}$ then $a_1-a\in \Ker(\Br_P)$ which means  $a_1-a=\Tr_Q^P(c)$ for some $c\in A_1^Q.$ Then
 \[(a_1-a)\otimes x=\Tr_Q^P(c)\otimes x=\Tr_Q^P(c\otimes x)\in \sum_{Q<P}(A_{\sigma})_Q^P,\]
 since $x\in C_H(P)$. So $\psi_{\tau}$ is a well-defined map.
 Moreover, the direct sum of these maps gives the graded homomorphism  \[\psi:R\to A'(P),\]   whose identity component is an isomorphism.
 Indeed, $A'(P)_1=A_1(P)$ and $R_1=A_1(P)\otimes 1\simeq A_1(P)$, and one can easily prove that $\psi_1$ is both injective and surjective.
The statement follows by applying \cite[Proposition 2.12]{Dade1}.
 \end{proof}

\section{The second extension} \label{s:second-ext}

\begin{subsec} Assume that the subgroup $P$ of $K$  is a defect group of $b$. According to \cite[\S 18]{The}, this means $b\in (A_1)_P^K$ and
$\Br_P(b)\neq 0.$ Since $b$ central and  primitive in $A_1^K$,  it is actually a point of $K$ on $A_1$  with defect group $P$ (see \cite[\S 3]{The}).
Applying  \cite[Theorem 3.2]{Barker}, the element $\bar{b}=\Br_P(b)$  lies in the center of $A_1(P)^{N_K(P)}$ and forms a singleton,
hence a point of $A_1(P)^{N_K(P)}$,
 also with defect group $P$. The identity component $R_1$ of the algebra $R$  constructed in \ref{s:r}  is isomorphic to $A_1(P).$ Using this isomorphism,
we regard  $\bar{b}$ as an element of $R_1$, so $\bar{b}$ is a point of $R_1^{N_K(P)}$ with defect group $P$.
\end{subsec}

\begin{subsec} \label{s:S} Let $C_H(P)_{\bar{b}}$ denote the stabilizer of $\bar{b}$ in $C_H(P)$, and let
 \[S:=\bigoplus_{\tau\in C_H(P)_{\bar{b}}/C_K(P)}R_{\tau}.\] Moreover $\bar{b}$ is a central element of $S$ and one can check the equality
 \[\bar{b}R\bar{b}:=\bigoplus_{\tau\in C_H(P)_{\bar{b}}/C_K(P)}\bar{b}R_{\tau}\]
 Then $\bar{b}R\bar{b}$ is     $C_H(P)/C_K(P)$-graded $N_H(P)_{\bar{b}}$-invariant algebra.
\end{subsec}

\begin{subsec}\label{f:g:d} Since $N_K(P)$ is a normal subgroup of $N_H(P)_{\bar{b}}$, we may consider the  centralizer
 \[C_S(C_K(P))^{N_K(P)}\cdot 1)=S^{N_K(P)}.\] The normal subgroup $N_K(P)$ of $N_H(P)_{\bar{b}}$ centralizes $C_H(P)_{\bar{b}}/C_K(P)$, and
   $S^{N_K(P)}$ is a  $N_H(P)_{\bar{b}}$-algebra in which $\bar{b}$ is  central. Then
 \[S^{N_K(P)}=\bigoplus_{ \tau \in C_H(P)_{\bar{b}}/C_K(P)}(R_{\tau})^{N_K(P)}, \mbox{ and }\]
 \[\bar{b}S^{N_K(P)}=(\bar{b}S)^{N_K(P)}=(\bar{b}R\bar{b})^{N_K(P)}=\bigoplus_{ \tau \in C_H(P)_{\bar{b}}/C_K(P)}\bar{b}(R_{\tau})^{N_K(P)}.\]
 Of course we need
 \[\overline{C_H(P)}[\bar{b}]=\{\tau\in C_H(P)_{\bar{b}}/C_K(P)\mid \bar{b}(R_{\tau})^{N_K(P)}\cdot \bar{b}(R_{\tau^{-1}})^{N_K(P)}=\bar{b}(R_{1})^{N_K(P)}\},\]
 the normal subgroup of $C_H(P)_{\bar{b}}/C_K(P)$ that makes $\bar{b}S^{N_K(P)}$ strongly graded.
 Denote
\begin{align*} D &:= \bigoplus_{\tau \in \overline{C_H(P)}[\bar{b}]}\bar{b}(R_{\tau})^{N_K(P)} \\
 &=\bigoplus_{\tau \in \overline{C_H(P)}[\bar{b}]}D_{\tau},\end{align*}
  where for each $\tau\in \overline{C_H(P)}[\bar{b}]$ we have \[D_{\tau}=\bar{b}(R_{\tau})^{N_K(P)}.\]
By construction, it follows that   $D_{\tau} D_{\tau^{-1}}= D_1$ for all $\tau\in \overline{C_H(P)}[\bar{b}]$.
 This makes $D$ a $\overline{C_H(P)}[\bar{b}]$-strongly graded $N_H(P)_{\bar{b}}$-algebra.
 By the localness of $D_1$ and by \cite[Lemma 1.1]{Sch} the order  $D$ is a crossed product of $D_1$ with
 $\overline{C_H(P)}[\bar{b}].$  \end{subsec}

\begin{subsec}\label{c:p:d} Denote \[\Bar{D}_1:= D_1/J(D_1)\] and $$\bar{D}=D/DJ(D_1).$$ Since the identity component
$D_1=\bar{b}R_1^{N_K(P)}$ of $D$ is a local ring, we obtain  $\bar{D}_1$  a skew field whose center $\hat{k}_2=Z(\bar{D}_1)$
is a finite extension of $k$. Moreover, $\bar{D}$ is an $N_H(P)_{\bar{b}}$-algebra, and in the same time a crossed product
of $\bar{D}_1$ with ${  \overline{C_H(P)}[\bar{b}]}$, corresponding to the group extension
\begin{align*}\tag{2}1\to \bar{D}_1^{*}\to \operatorname{hU}(\bar{D})\to \overline{C_H(P)}[\bar{b}]\to 1.\end{align*}
As usual, for $\tau\in \overline{C_H(P)}[\bar{b}]$ we identify $$\bar{D}_{\tau}=D_{\tau}/D_{\tau}J(D_1).$$
\end{subsec}

\section{The third extension}\label{s:third-ext}
If, as above, $N_H(D)_b$ denotes the stabilizer of $b$ in $N_H(P)$ and  $N_H(D)_{\bar{b}}$ denotes the stabilizer of $\bar{b}$ in $N_H(P),$ one can easily show that these two stabilizers are equal.

\begin{subsec}  We return to our $G$-graded $H$-interior algebra $$A=\bigoplus_{\sigma\in G}A_{\sigma},$$
 and to its $H_b$-subalgebra
\[bA:=\bigoplus_{\sigma\in G_b}bA_{\sigma}.\] Since the defect group $P$ fixes $b$ we have $(bA)^P=bA^P$ and we may consider the Brauer homomorphism
in the following situation.
\[\Br_P:(bA)^P\to (bA)(P).\] Of course this is a morphism of $N_H(P)_b$-algebras.
Note that we have the direct sums decompositions
\[(bA)^P=\bigoplus_{\sigma\in G_b}(bA_{\sigma})^P\] and  \[(bA)(P)=\bigoplus_{\sigma\in G_b}((bA_{\sigma})^P/\sum_{Q<P}(bA_{\sigma})^P_Q):=
\bigoplus_{\sigma\in G_b}(bA)(P)_{\sigma}.\] This is actually the same situation as in Section \ref{s:brauer}, so we do not give any further explanation on
the structure of these algebras.
\end{subsec}

\begin{subsec}
We are going to relate this Brauer quotient to the construction made in Section \ref{s:second-ext}. For that we need to make one replacement.
Recall that $\bar{b}$ denotes the image of $b$ under the Brauer morphism evaluated on the identity components. We see that for any $\sigma \in G_b$ there is an
isomorphism of $N_H(P)_b$-invariant $k$-spaces, that is
\[(bA_{\sigma})^P/\sum_{Q<P}(bA_{\sigma})^P_Q\simeq \bar{b}((A_{\sigma})^P/\sum_{Q<P}(A_{\sigma})^P_Q).\] Thus we can reconsider the Brauer morphism in this way:
\[\Br_P:bA^P\to \bar{b}A''(P)=\bar{b}A(P)\bar{b},\] where
\[A''(P)=\bigoplus_{\sigma\in G_b}A(P)_{\sigma}.\]
\end{subsec}

\begin{subsec}
Next we are interested in the subgroups of $G_b$ for which $bA^P$ and $\bar{b}A''(P)$ become strongly graded. So we consider
\[T_1=\{\sigma\in G_b\mid bA_{\sigma}^P\cdot bA_{\sigma^{-1}}^P=bA_{1}^P\}\]  and
\[ T_2=\{\sigma\in G_b\mid \bar{b}A''(P)_{\sigma}\cdot \bar{b}A''(P)_{\sigma^{-1}}=\bar{b}A(P)_{1}\}.\] Since we are concerned
mainly with subalgebras that are strongly graded, as we will see, all the subgroups of $G_b$ that appear in
our further constructions are included in $T_1$ or in $T_2.$ An obvious remark is that the Brauer morphism carries the surjection
componentwise, and this implies that the restriction
\[\bigoplus_{\sigma\in T_1}(bA_{\sigma})^P \to \bigoplus_{\sigma\in T_2}\bar{b}A(P)_{\sigma}\] is well defined. Indeed, if $\sigma\in T_1$
then we have
$$\bar{b}A(P)_1=\Br_P(bA_1^P)=\Br_P(bA_{\sigma}^P\cdot bA_{\sigma^{-1}}^P)=\bar{b}A''(P)_{\sigma}\cdot \bar{b}A''(P)_{\sigma^{-1}}.$$ This
restriction could be surjective provided that the idempotent $b$ remains primitive in $A_1^P,$
which in general it is not true.

Following the proof of Proposition \ref{proposition 2.4} one can show that $T_1$ and $T_2$
are both  respectively $N_H(P)_b$-invariant and  $N_H(P)_{\bar{b}}$-invariant  subgroups of $G_b.$
But since $N_H(P)_b=N_H(P)_{\bar{b}}$ and $T_1\leq T_2$ this restriction is a morphism of $N_H(P)_b$-algebras.
 \end{subsec}

\begin{subsec}\label{s:ext3}
At the beginning of Section \ref{s:second-ext} we saw that $\bar{b}$ is a primitive central idempotent of $A(P)_1^{N_K(P)}\simeq A_1(P)^{N_K(P)},$
hence $\bar{b}$ is a central idempotent of  \[\bar{b}A''(P)^{N_K(P)}=\bigoplus_{\sigma\in G_b}\bar{b}A(P)_{\sigma}^{N_K(P)}\] and of
\[(\bigoplus_{\sigma\in T_2}\bar{b}A(P)_{\sigma})^{N_K(P)}=\bigoplus_{\sigma\in T_2}\bar{b}A(P)_{\sigma}^{N_K(P)}.\]
In both cases we can introduce the subgroups
\[G[\bar{b}]=\{\sigma\in G_b\mid \bar{b}A(P)_{\sigma}^{N_K(P)}\cdot \bar{b}A(P)_{\sigma^{-1}}^{N_K(P)}=\bar{b}A(P)_{1}^{N_K(P)}\}\]  and
\[T_2[\bar{b}]=\{\sigma\in T_2\mid \bar{b}A(P)_{\sigma}^{N_K(P)}\cdot \bar{b}A(P)_{\sigma^{-1}}^{N_K(P)}=\bar{b}A(P)_{1}^{N_K(P)}\}\] of $G_b$, determining
two strongly graded subalgebras. As expected, these  subgroups coincide and are $N_H(P)_b$-invariant.

The last statement follows by using the localness of $\bar{b}A(P)_{1}^{N_K(P)},$ the inclusions
\[(\bigoplus_{\sigma\in T_2}\bar{b}A(P)_{\sigma})^{N_K(P)}\hookrightarrow \bar{b}A''(P)^{N_K(P)}\hookrightarrow \bar{b}A''(P)\] and \cite[Lemma 1.1]{Sch}. Then for any
$\sigma\in T_2[\bar{b}]$ there is a unit $\bar{a}_{\sigma}\in \bar{b}A(P)_{\sigma}^{N_K(P)}\cap \mathrm{U}(\bar{b}A''(P)^{N_K(P)})$ such that
$$\bar{b}A(P)_{\sigma}^{N_K(P)}=\bar{a}_{\sigma}\bar{b}A(P)_{1}^{N_K(P)}=\bar{b}A(P)_{1}^{N_K(P)}\bar{a}_{\sigma}.$$ So $T_2[\bar{b}]$ is a subgroup of
$G[\bar{b}].$ Now let $\sigma\in G[\bar{b}].$ A similar argument as before gives a unit
$\bar{a}_{\sigma}\in \bar{b}A(P)_{\sigma}^{N_K(P)}\cap \mathrm{U}(\bar{b}A''(P)^{N_K(P)})$
such that \[\bar{b}A(P)_{\sigma}^{N_K(P)}=\bar{a}_{\sigma}\bar{b}A(P)_{1}^{N_K(P)}=\bar{b}A(P)_{1}^{N_K(P)}\bar{a}_{\sigma},\] as a component of
$\bar{b}A''(P)^{N_K(P)}.$ Using one of the above inclusion
 it follows that $\bar{a}_{\sigma}$ is an invertible element in the biggest algebra. This implies that $\sigma$ belongs to $T_2.$
Keeping in mind that $\bar{a}_{\sigma}$ is still an $N_K(P)$-invariant invertible element, we have $\sigma\in T_2[\bar{b}].$ The $N_H(P)_b$-invariance
follows again using the technique of the proof of Proposition \ref{proposition 2.4}.
\end{subsec}

\begin{subsec}
In the previous paragraph we saw that regardless the starting point, we end up with the same subgroup that makes our $N_K(P)$-invariant
subalgebra of $\bar{b}A(P)''$ strongly graded. Repeating the construction of Section \ref{s:second-ext} we obtain the algebra
\[E:=\bigoplus_{\sigma\in G[\bar{b}]}\bar{b}A(P)_{\sigma}^{N_K(P)}:=\bigoplus_{\sigma\in G[\bar{b}]}E_{\sigma},\] which is a crossed product of $E_1:=\bar{b}A(P)_{1}^{N_K(P)}$ with $G[\bar{b}]$ such that
$\bar{E_1}:=E_1/J(E_1)$ is a skew field whose center $\hat{k}_3=Z(\bar{E_1})$ is a finite extension of $k.$ The quotient
$\bar{E}:=E/EJ(E_1)$ is a $N_H(P)_b$-algebra, a crossed product of $\bar{E_1}$ with $G[\bar{b}]$ corresponding to the Clifford extension
\begin{align*}\tag{3}1\to \bar{E}_1^{*}\to \operatorname{hU}(\bar{E})\to G[\bar{b}]\to 1.\end{align*}
We should note  that $E$ does not contain any zero components, meaning that for any $\sigma\in G[\bar{b}]$ we have $\bar{b}A(P)_{\sigma}^{N_K(P)}\neq 0.$
\end{subsec}

\section{Remarks on the first extension}

We  go back to our first extension, since we have introduced the group $T_1$, and we need to relate it to $G[b].$
\begin{subsec}
First, let us observe that $G[b]$ is a subgroup of $T_1.$ This follows from the fact that $C$ is a crossed product (see \ref{s:c}) and because of
the inclusions
\[C=\bigoplus_{\sigma\in G[b]}b(A_{\sigma})^K\hookrightarrow  \bigoplus_{\sigma\in G_b}b(A_{\sigma})^K\hookrightarrow \bigoplus_{\sigma\in G_b}b(A_{\sigma})^P.\]
\end{subsec}

\begin{subsec}
Conversely, as in the previous section, since
\[\bigoplus_{\sigma\in T_1}b(A_{\sigma})^K\hookrightarrow \bigoplus_{\sigma\in T_1}b(A_{\sigma})^P,\] we have $T_1[b]=G[b].$
\end{subsec}

\begin{subsec}

If $\sigma\in G[b]$ since $bA_1^K=bA_{\sigma}^K\cdot bA_{\sigma^{-1}}^K$ then $bA_{\sigma}^K\neq 0$ and $bA_{\sigma}^K\nsubseteq \sum_{Q<P}(bA_{\sigma})^P_Q.$ This easily follows from the definition of $G[b]$ and from the proof of Lemma \ref{ext-iso} below. This remark forces the corresponding component in the Brauer quotient to satisfy
$\bar{b}A(P)_{\sigma}^{N_K(P)}\neq 0.$
\end{subsec}
\section{The isomorphism of  two  extensions}

We keep the notations of the previous sections. The next  lemma will be needed.
\begin{lemma}\label{ext-iso} The skew fields $\bar{C}_1$ and $\bar{E}_1$ are isomorphic as $N_H(P)_b$-algebras.
\end{lemma}

\begin{proof}
Recall that \[\bar{C}_1= C_1/J(C_1)=bA_1^K/J(bA_1^K)\] and
 \[\bar{E}_1= E_1/J(E_1)= {\bar{b}}(A_1(P))^{N_K(P)}/J({\bar{b}}(A_1(P))^{N_K(P)}).\]
  We have  $P$  a defect group of the point $\{b\} \subset A_1^K$, and $b\in (A_1)_P^K$.
  Arguments similar to those of   \cite[Lemma 3.4]{Rha},
  together with the proof of   \cite[Lemma 1.12]{BrouePuig},
  show that the map \[\Br_P:bA_1^K\to (\Br_P(b)A_1(P))^{N_K(P)}=\bar{b}(A_1(P))^{N_K(P)}\] is onto.
  Notice that $$\Br_P((A_1)_P^K)=A_1(P)_P^{N_K(P)}.$$ This is true since for $\Tr_P^K(a)\in (A_1)_P^K,$
  by using the Mackey decomposition, we get
 $$\Br_P(\Tr_P^K(a))=\Br_P(\sum_{l\in P\setminus K/P}\Tr_{P\cap P^l}^P(a^l))=\sum_{l\in N_K(P)/P}(\Br_P(a))^l.$$
  The idempotent $b$ belongs to $(A_1)_P^K$, and the ideal $b(A_1)_P^K$ is mapped onto the ideal $\bar{b}A_1(P)_P^{N_K(P)},$
 which  contains the identity $\bar{b}$ of $\bar{b}A_1(P)^{N_K(P)}$. To finish the proof, it suffices to apply  \cite[Proposition 3.23]{Puig}.
\end{proof}

Before stating the main result we introduce one more subgroup of $T_2.$
\begin{subsec}\label{lastsubgroupT_2}
The $N_H(P)_b$-subalgebra
\[\bigoplus_{\sigma\in N_H(P)_bK/K}\bar{b}A(P)_{\sigma}^{N_K(P)}\] need not be a strongly graded subalgebra of $\bar{b}A''(P).$ So we introduce the $N_H(P)_{\bar{b}}$-invariant subgroup
\[\overline{N_H(P)}[\bar{b}]=\{\sigma\in N_H(P)_bK/K\mid \bar{b}A(P)_{\sigma}^{N_K(P)}\cdot \bar{b}A(P)_{\sigma^{-1}}^{N_K(P)}=\bar{b}A(P)_{1}^{N_K(P)}\}\] of $N_H(P)_bK/K$.
The inclusions
\[\bigoplus_{\sigma\in \overline{N_H(P)}[\bar{b}]}\bar{b}A(P)_{\sigma}^{N_K(P)}\hookrightarrow \bigoplus_{\sigma\in N_H(P)_bK/K}\bar{b}A(P)_{\sigma}^{N_K(P)}
\hookrightarrow \bar{b}A''(P),\]
show that $\overline{N_H(P)}[\bar{b}]$ is a subgroup of $T_2.$
\end{subsec}

\begin{theorem} \label{3}  The following statements hold.

\begin{enumerate}
\item[$(1)$] $G_b$ equals $N_H(P)_{b}K/K=N_H(P)_{\bar{b}}K/K$.
\item[$(2)$] The groups $G[b]$ and $G[\bar{b}]$ are equal and they both coincide with $\overline{N_H(P)_b}[\bar{b}]$.
\item[$(3)$] The extensions $(1)$ and $(3)$ are isomorphic.
\item[$(4)$] The isomorphism between the extensions $(1)$ and $(3)$ is compatible with the identity isomorphism
\[G[b]\to \overline{N_H(P)}[\bar{b}], \] and preserves the action  of
$$G_b\simeq N_H(P)_{\bar{b}}/N_K(P)$$ on the two extensions.
\item[$(5)$] There is a  monomorphism from extension $(2)$ into extension $(3)$, which is also compatible with the natural monomorphism
$$\overline{C_H(P)}[\bar{b}]\to G[b],$$ and preserves the action of $G_b$ on these extensions.
\end{enumerate}
\end{theorem}

\begin{proof} We know that the Brauer morphism is compatible with the $N_H(P)$-action.  We also know that $\bar{b}=\Br_P(b)$ and $N_H(P)_b=N_H(P)_{\bar{b}}.$
Denote by $H_b$ the inverse image of $G_b$ in $H$. Then $H_b$ normalizes $K$, and since all defect groups of $b$
are conjugate under $K$, $H_b$ acts on the set of defect groups of $b$.  We have  $b\in (A_1)_P^K$ and $\Br_P(b)\neq 0.$ Then
if $h\in H_b$, we obtain $b\in (A_1)_{P^h}^K$, and clearly $\Br_{P^h}(b)\neq 0$, because otherwise $b\in \sum_{Q<P^h}(A_1)^{P^h}_Q$,
 implying $\Br_P(b)=0.$ So for any $h\in H_b$ there is $l\in K$ such that $P^h=P^l$, and from this we obtain $hl^{-1}\in N_{H_b}(P)=N_H(P)_b.$
 The  inclusion $N_H(P)_bK\subseteq H_b$ is trivial, proving the equality \[H_b=N_H(P)_bK=N_H(P)_{\bar{b}}K.\] This  is our  first statement.

We claim that the restriction of the Brauer homomorphism
\[(bA)^K=\bigoplus_{\sigma\in G_b}(bA_{\sigma})^K\to \bar{b}A''(P)^{N_K(P)}=\bigoplus_{\sigma\in G_b}\bar{b}A(P)_{\sigma}^{N_K(P)} \]
is an epimorphism. Indeed, because $P$ is a defect group of $b$, this block belongs to the ideal $(A_1)_P^K\subseteq  A_P^K.$ The equality
\[\Br_P(bA_P^K)=\bar{b}A''(P)_P^{N_K(P)}\]  is well known.
The sets $bA_P^K$ and $\bar{b}A''(P)_P^{N_K(P)}$ are ideals, one in $bA^K$ and the other in  $\bar{b}A''(P)^{N_K(P)},$
both containing the identity element of the respective algebra, so the claim is proved.

If $\sigma\in G[b]$ then $bA_1^K=bA_{\sigma}^K\cdot bA_{\sigma^{-1}}^K.$ Using the proof of Lemma \ref{ext-iso} we obtain
\begin{align*}\bar{b}A(P)_1^{N_K(P)}&=\Br_P(bA_1^K)=\Br_P(bA_{\sigma}^K)\Br_P(bA_{\sigma^{-1}}^K) \\ &=\bar{b}A(P)_{\sigma}^{N_K(P)}\bar{b}A(P)_{\sigma^{-1}}^{N_K(P)}
=E_{\sigma}\cdot E_{\sigma^{-1}}, \end{align*}
so the image of  $C$ is in $E$, and $G[b]\leq G[\bar{b}].$

If $\sigma\in G[\bar{b}]$ we have $$E_{\sigma}\cdot E_{\sigma^{-1}}=E_1.$$  Then
$$\Br_P^{-1}(E_{\sigma})\cdot \Br_P^{-1}(E_{\sigma^{-1}})\nsubseteq J(C_1),$$ because otherwise
$$E_{\sigma}\cdot E_{\sigma^{-1}}=\Br_P(\Br_P^{-1}(E_{\sigma})\cdot \Br_P^{-1}(E_{\sigma^{-1}}))\subseteq \Br_P(J(C_1))\subseteq J(E_1)$$ which is false.
We have
$$\Br_P^{-1}(E_{\sigma})\cdot \Br_P^{-1}(E_{\sigma^{-1}})=bA^K_{\sigma}\cdot bA^K_{\sigma^{-1}}\subseteq C_1,$$  and we obtain
$$J(C_1)+ bA^K_{\sigma}\cdot bA^K_{\sigma^{-1}}=C_1.$$ Consequently, $\sigma \in G[b]$, and this proves the inclusion $G[\bar{b}]\leq G[b].$ Together with
this inclusion we have shown that the restriction
\[\Br_P:C\to E\] is also an epimorphism of $N_H(P)_b$-algebras.

We have already used the inclusion $\Br_P(J(C_1))\subseteq J(E_1)$, which is a well-known result on $\mathcal{O}$-algebras
related by an epimorphism.
  Using the fact that $C$ and $E$ are strongly graded algebras, the results in \cite[1.5.A.]{M} prove
$$\Br_P(CJ(C_1))\subseteq EJ(E_1).$$
In this way we get a new $N_H(P)_b$-algebra epimorphisms,  namely
$$\overline{\Br_P}:\bar{C}\to \bar{E}, \mbox{ for } \bar{a}\in \bar{C} \mbox{ we have } \overline{\Br_P}(\bar{a})=\overline{\Br_P(a)}.$$
Statement $(3)$ follows from this $G[b]$-graded epimorphism, from the equality $G[b]=G[\bar{b}]$ and Lemma \ref{ext-iso}.

Using $(1)$ the inclusion
\[\bigoplus_{\sigma\in N_H(P)_bK/K}\bar{b}A(P)_{\sigma}\hookrightarrow \bar{b}A''(P)=\bigoplus_{\sigma\in G_b}\bar{b}A(P)_{\sigma}\]
is now an equality. Forcing
\[\bigoplus_{\sigma\in N_H(P)_bK/K}\bar{b}A(P)_{\sigma}^{N_K(P)}\hookrightarrow \bar{b}A''(P)^{N_K(P)}=\bigoplus_{\sigma\in G_b}\bar{b}A(P)_{\sigma}^{N_K(P)}\]
to be an equality. The definitions of $G[\bar{b}]$ and $ \overline{N_H(P)}[\bar{b}],$ see \ref{s:ext3} and \ref{lastsubgroupT_2}, prove the equality of
these two groups. This shows the last part of $(2).$

By the construction of the first extension $G[b]$ is a normal subgroup of $G_b$, so by using assertions $(1)$ and $(2),$ $G[\bar{b}]$ is a normal subgroup
of $N_H(P)_bK/K.$ Since the Brauer map is a morphism of $N_H(P)_b$-algebras, statement (4) of the theorem is immediate.

By Proposition \ref{brauer:iso}, the $N_H(P)$-algebras
\[R:=\bigoplus_{\tau\in C_H(P)/C_K(P)}R_{\tau}\]  and
\[A'(P)=\bigoplus_{\sigma\in C_H(P)K/K}A(P)_{\sigma}\]
are isomorphic as $N_H(P)$-algebras. So
$R^{N_K(P)} \mbox{ and } A'(P)^{N_K(P)}$
are isomorphic as $N_H(P)$-algebras too.
This implies that
\[(\bar{b}R\bar{b})^{N_K(P)}=\bar{b}S^{N_K(P)}\overset{\psi}{\simeq} \bar{b}A'(P)^{N_K(P)}\]  and
\[\bar{b}A'(P)^{N_K(P)}=\bigoplus_{\sigma\in C_H(P)_{\bar{b}}K/K}\bar{b}A'(P)^{N_K(P)} _{\sigma}.\]
Using the isomorphism $\psi,$ the group $\overline{C_H(P)}[\bar{b}]$ is isomorphic to its analogous  subgroup of
  $C_H(P)_{\bar{b}}K/K.$
Indeed, if $\overline{C_H(P)}[\bar{b}]'$ denotes the subgroup of $C_H(P)_{\bar{b}}K/K$ isomorphic to $\overline{C_H(P)}[\bar{b}]$
under \[C_H(P)_{\bar{b}}/C_K(P)\simeq C_H(P)_{\bar{b}}K/K\] then  $\overline{C_H(P)}[\bar{b}]'$ is the largest subgroup for which the
restriction of  $\bar{b}A'(P)^{N_K(P)}$ is
strongly graded.
If $\tau\in \overline{C_H(P)}[\bar{b}]$ then
\begin{align*}\bar{b}A(P)_1^{N_K(P)}=\psi(\bar{b}R_1^{N_K(P)})&=\psi(\bar{b}R_{\tau}^{N_K(P)})\cdot \psi(\bar{b}R_{\tau^{-1}}^{N_K(P)})\\
&=\bar{b}A(P)_{\sigma}^{N_K(P)}\cdot
\bar{b}A(P)_{\sigma^{-1}}^{N_K(P)},\end{align*}
where $\sigma\cap C_H(P)_{\bar{b}}=\tau$ and $\sigma$ belongs to  $\overline{C_H(P)}[\bar{b}]'.$
Conversely, if $\sigma\in \overline{C_H(P)}[\bar{b}]',$ as before
\[\psi^{-1}(\bar{b}A(P)_{\sigma}^{N_K(P)})\cdot \psi^{-1}(\bar{b}A(P)_{\sigma^{-1}}^{N_K(P)})\nsubseteq J(\bar{b}R_1^{N_K(P)}),\]
since $\bar{b}R_1^{N_K(P)}$ is local and $\psi(J(\bar{b}R_1^{N_K(P)}))=J(\bar{b}A_1(P)^{N_K(P)}).$
Hence for the corresponding $\tau=\sigma\cap C_H(P)_{\bar{b}}$ we have
\[\bar{b}R_1^{N_K(P)}=\bar{b}R_{\tau}^{N_K(P)}\cdot \bar{b}R_{\tau^{-1}}^{N_K(P)}.\]
By the definition of the action of $N_H(P)_b$ on these algebras, the groups $\overline{C_H(P)}[\bar{b}]'$ and $\overline{C_H(P)}[\bar{b}]$
are both $N_H(P)_b$-invariant hence normal subgroups of $N_H(P)_bK/K$ and of $N_H(P)_b/C_K(P)$ respectively. So $\overline{C_H(P)}[\bar{b}]$ embeds into
$G[b]$ and then $D$ embeds in $E$. The equalities
\[J(D)=DJ(D_1), J(E)=EJ(E_1) \mbox{ and } J(E_1)=J(D_1)\] define a map
$\bar{D}\to \bar{E}$ such that all the vertical maps in the
 following commutative diagram
$$
\begin{xy}
\xymatrix{
 1 \ar[r]  &\bar{D}_1^{*}\ar[r]\ar[d] &\operatorname{hU}(\bar{D})\ar[r]\ar[d] &\overline{C_H(P)}[\bar{b}]\ar[r]\ar[d] &1\\
 1\ar[r] &\bar{E}_1^{*}\ar[r] &\operatorname{hU}(\bar{E})\ar[r] &G[\bar{b}]\ar[r] &1\\
}
\end{xy}
$$
are injective.
\end{proof}

\section{The group algebra case} \label{s:groupalg}

\begin{subsec}
The main theorem shows that the group $G[b]$ defining the Clifford extension equals $\overline{N_H(P)}_b[\bar{b}]$ and it is at least $\overline{C_H(P)}[\bar{b}].$
This situation is generated by the groups $T_1$ and $T_2$. They are the first that strongly graduate the Brauer domain and codomain, and they include all the others
subgroups of the algebras we work with.
Actually $G[b]$ equals  $\overline{C_H(P)}[\bar{b}]$ exactly when  $\overline{C_H(P)}[\bar{b}]=\overline{N_H(P)}[\bar{b}]$.

There are some cases of $K$-interior $H$-algebras for which $T_2$ is included or it equals the centralizer. This is the case of a group algebra.
Let $G$ denote the quotient $H/K,$ let $A=\mathcal{O}H$ and let $b$ be a block of $\mathcal{O}K$ having defect group $P\leq K.$
The special case of the group algebra gives
\[A=\bigoplus_{\sigma\in G}A_{\sigma}, \mbox{ where } A_{\sigma}=\mathcal{O}\sigma, \mbox{ for all } \sigma\in G\]
and $b$ is primitive in $Z(\mathcal{O}K)=\mathcal{O}K^K.$ For each $\sigma\in G$ we denote by
  $C_{\sigma}$ the intersection $b(\mathcal{O}\sigma\cap A^K).$
As above we introduce $G[b]$ and $G_b,$ and then
\[bAb=\bigoplus_{\sigma\in G_b}bA_{\sigma}=b\mathcal{O}H_b,\] while
\[C=\bigoplus_{\sigma\in G[b]}bC_{\sigma}\] is a strongly $G[b]$-graded $H_b$-algebra. Letting $\hat{k}_1=C_1/J(C_1)$, the quotient
\[C/CJ(C_1)=\bigoplus_{\sigma\in G[b]}C_{\sigma}/C_{\sigma}J(C_1)\] is the crossed product of $\hat{k}_1$ with $G[b]$ that corresponds
to the Clifford extension
\begin{align*}\tag{1'}1\to \hat{k}_1^{*}\to \operatorname{hU}(\bar{C})\to G[b]\to 1.\end{align*}
\end{subsec}
\begin{subsec}
If $b_1$ is the Brauer correspondent of $b$, it also has defect $P$ and it lies in $$Z(\mathcal{O}N_K(P))=\mathcal{O}N_K(P)^{N_K(P)}.$$
The $N_K(P)$-interior $N_H(P)$-algebra $\mathcal{O}N_K(P)$ is the identity component of
\[\mathcal{O}N_H(P)=\bigoplus_{\sigma\in N_H(P)/N_K(P)}\mathcal{O}\sigma.\]
Moreover
\[b_1\mathcal{O}N_H(P)b_1=b_1\mathcal{O}N_H(P)_{b_1}=\bigoplus_{\sigma\in N_H(P)_{b_1}/N_K(P)}b_1\mathcal{O}\sigma.\]
As in the above construction, we have the  normal subgroup $G[b_1]$ of $N_H(P)_{b_1}/N_K(P)$ determining a strongly graded $N_H(P)_{b_1}$-subalgebra
of \[b_1\mathcal{O}N_H(P)_{b_1}^{N_K(P)}=\bigoplus_{\sigma\in N_H(P)_{b_1}/N_K(P)}b_1(\mathcal{O}\sigma)^{N_K(P)};\]
more precisely,
\[E=\bigoplus_{\sigma\in G[b_1]}E_{\sigma}=\bigoplus_{\sigma\in G[b_1]}b_1(\mathcal{O}\sigma)^{N_K(P)}.\]
Then letting $\hat{k}_2=E_1/J(E_1)$, the quotient $\bar{E}=E/EJ(E_1)$ is the crossed product of $\hat{k}_2$ with $G[b_1]$ associated to the extension
\begin{align*}\tag{2'}1\to \hat{k}_2^{*}\to \operatorname{hU}(\bar{E})\to G[b_1]\to 1.\end{align*}
\end{subsec}

\begin{subsec}
Let us take a look of the Brauer quotients of these two algebras. We have
\[\Br^H_P:(\mathcal{O}H)^P\to (\mathcal{O}H)^P/\sum_{Q<P}(\mathcal{O}H)_Q^P\simeq kC_H(P)\] and
\[\Br^{N_H(P)}_P:(\mathcal{O}N_H(P))^P\to (\mathcal{O}N_H(P))^P/\sum_{Q<P}(\mathcal{O}N_H(P))_Q^P\simeq kC_H(P).\]
Then we have the maps
\[\Br^H_P:(b\mathcal{O}H_b)^K\to  \bar{b}kC_H(P)_{\bar{b}}^{N_K(P)}\]  and
\[\Br^{N_H(P)}_P:(b_1\mathcal{O}N_H(P)_{b_1})^{N_K(P)}\to \bar{b}_1kC_H(P)_{\bar{b}_1}^{N_K(P)},\]
where $\bar{b}:=\Br^H_P(b),$  $\bar{b}_1:=\Br^{N_H(P)}_P(b_1),$ $N_H(P)_b=N_H(P)_{\bar{b}}$ and $N_H(P)_{b_1}=N_H(P)_{\bar{b}_1}.$ Since
$\bar{b}=\bar{b}_1$ (see  the proof of \cite[Theorem 5.1]{AschKessOliv}), the $N_H(P)_b$-algebra
$$D:=\bar{b}kC_H(P)_{\bar{b}}^{N_K(P)}$$ determines a unique extension
\begin{align*}\tag{3'}1\to \hat{k}_3^{*}\to \operatorname{hU}(\bar{D})\to \overline{C_H(P)}[\bar{b}]\to 1.\end{align*}  By
applying twice the main result, once for each of the Brauer maps defined above, we get that the extensions $(1')$, $(2')$ and $(3')$ are  isomorphic. Note that in the case of the group algebra we have
\[(b\mathcal{O}H_b)(P)=(b\mathcal{O}N_H(P)_bK)(P)=(b_1\mathcal{O}N_H(P)_{b_1})(P)=\bar{b}kC_H(P)_{\bar{b}},\] and this is why
$G[b]=G[\bar{b}]=G[b_1]=\overline{N_H(P)_b}[\bar{b}]=\overline{C_H(P)_b}[\bar{b}].$
\end{subsec}

\end{document}